\documentclass[10pt]{article}
\usepackage{hyperref}
\usepackage[english]{babel}
\usepackage[utf8]{inputenc}
\usepackage{amsmath,amssymb,enumerate}
\usepackage{latexsym}
\usepackage{latexsym}
\usepackage{amsfonts}
\usepackage{mathrsfs}
\usepackage{amstext}
\usepackage{amsthm}
\usepackage{amssymb}
\usepackage{amscd}
\theoremstyle{definition}
\newtheorem{ejem}{Example}[section]
\newtheorem{defi}{Definition}[section]
\theoremstyle{plain}
\newtheorem{teo}{Theorem}[section]
\newtheorem{prop}[teo]{Proposition}
\newtheorem{lema}[teo]{Lemma}
\newtheorem{corol}[teo]{Corollary}

\newcommand{\re}{\mathbb{R}}

\newcommand{\nat}{\mathbb{N}}

\newcommand{\sgn}{\operatorname{sign}}
\newcommand{\SSh}{\mathcal{S}}

\title{On the Cauchy problem associated to
 a regularized Benjamin-Ono -- Zakharov–Kuznetsov (rBO-ZK) type equation.}
\author{Fabián Sánchez S. and F\'elix H. Soriano M.}
\linethickness{.12mm}
\begin{document}\maketitle
\begin{abstract}
  In this work we shall study the well-posedness and ill-posedness of
  the Cauchy problem associated to the equation
\begin{equation*} 
u_{t}+a(u^{n})_{x}+(b\mathscr{H} u_{t}+u_{yy})_{x}=0, 
\end{equation*} 
in anisotropic weigthed Sobolev spaces. In particular, we examine the
unique continuation property for solutions, of this, with polynomial
decay.
\end{abstract}
\section{Introduction} 
The non-linear evolution equations play a very important role in
different areas of science and engineering. It is worth mentioning
some of them: fluid mechanics, plasma physics, fiber optics, solid
state physics, chemical kinetics, chemical physics and geochemistry,
among others. So, starting from the study of their solutions, it is
understood the effects of dispersion, diffusion, reaction and
convection associated with the models described by these. For example,
Korteweg-de Vries equation 
\begin{equation}\label{KdV}
u_{t}=u_{xxx}+uu_{x} \qquad(x,t)\in \mathbb{R}^2,
\end{equation}
that models the behavior of waves in shallow water channels, has
solitary waves as solutions which behave as particles, in the sense
that after they collide the shape and velocity are retained, which was
the reason to Kruskal and Zabusky coined the term \emph{soliton} in
their work of 1965 (see \cite{zbkru}). These solitons are stable, in
the sense that if a solution initially differs little in shape to
soliton type solutions, along the time, this will maintain its shape
to differ very little from that of the solution type soliton (see
\cite{Ben72} and \cite{Bona75}); indeed, this solution eventually
takes the soliton shape (see \cite{pewe}). For practical purposes the
notion stability of solitons guarantees, taking meticulous care, that in the
laboratory it can be reproduced these. It was first done by J. Scott Russel
in 1934.\par The Benjamin-Bona-Mahony equation
\begin{equation}\label{BBM}
u_{t}+u_{x}+uu_{x}-u_{xxt}=0,
\end{equation}
was introduced in \cite{BBM} with the intention to model the
propagation of long waves of small amplitude, where the effect is
purely nonlinear dispersion. The way this was obtained, it was sought
to reach an equation equivalent to the KdV equation
\eqref{KdV}. Interestingly, despite this intention, from the purely
mathematical point of view, these equations have significant and
interesting differences.\par Other one-dimensional equations are, one,
the introduced independently by Benjamin in \cite{Benjamin} and by Ono
in \cite{ono},
\begin{equation}\label{BO}
u_{t}+\mathscr{H}u_{xx}+uu_{x}=0.
\end{equation}
which models the internal waves in stratified fluids deep, where
$\mathscr{H}$ is the Hilbert transform.
The another, is the regularized Benjamin-Ono equation
\begin{equation}\label{rBO}
u_{t}+u_{x}+uu_{x}+\mathscr{H}u_{xt}=0
\end{equation}
where $u=u(x,t)$ is a real valued function, with $x,t \in
\mathbb{R}$. This equation is a model for the evolution in time of
waves with large crests at the interface between two fluids
inmiscibles. \par
There are bidimensional versions extending the above equations. In
the case of the KdV equation we have the Kadomtsev-Petviashvilli
equation, see \cite{KP},
\begin{equation}\label{KP}
(u_{t}+auu_{x}+u_{xxx})_{x}+u_{yy}=0,
\end{equation}
which describes waves in thin films of high surface tension. Another
is the Zakharov-Kuznetsov equation
\begin{equation}\label{ZK}
u_{t}=(u_{xx}+u_{yy})_{x}+uu_{x},
\end{equation}
which arises in the study of geophysical fluid dynamics in isotropic
sets (media in which the characteristics of the bodies does not depend
on the direction) and ion acoustic waves in magnetic plasmas.
\par For the case of the equations BBM \eqref{BBM},  BO \eqref{BO} and
rBO \eqref{rBO} we have the following  bidimensionals versions, the
ZK-BBM equation 
\begin{equation}\label{ZK-BBM}
 u_{t}=(u_{xt}+u_{yy})_{x}+ uu_{x},
\end{equation}
the ZK-BO equation
\begin{equation}\label{ZK-BO}
 u_{t}=(\mathcal Hu_{x}+u_{yy})_{x}+ uu_{x},
\end{equation}
and the rZK-BO equation
\begin{equation}\label{eq:rBOZK}
\begin{cases} 
u_{t}+a(u^{n})_{x}+(b\mathscr{H}u_{t}+u_{yy})_{x}=0, \quad (x,y) \in \re^2 , t>0 \\ u(0,x,y)=\varphi(x,y) \end{cases}
\end{equation} 
\par There are important questions that can be done about these
equations, and whose answers help clarify the phenomena that model
each of these. Let us note the following three questions: the
well-posedness (local and global), the existence and stability of
traveling waves and the unique continuation. In this direction, KdV,
BBM, BO, KP and ZK equations have deserved a comprehensive study.  In
\cite{BonaKalish}, \cite{mammeri} and \cite{Angulo} was studied the
well-posedness of the rBO equation. In \cite{Angulo} it was also
examined the existence and stability of travelling periodic waves of
rBO equation. In \cite{FonsecaGui} was studied the well and
ill-posedness of rBO in weighted Sobolev spaces, in particular, they
obtained a result on unique continuation property of this equation
that shows the no persitence of solutions of this in spaces of
functions with arbitrary decay polinomial. Scarcer is the literature
on the ZK-BBM equation. Of this we can say that, very recently, the
local well-poseness has been studied in \cite{Vega}. The existence and
stability of solitary waves were examined in \cite{Johnson},
\cite{Wazwaz} and \cite {Wazwaz2}. In \cite{Wazwaz} (actually, in the
literature the same author cited in this article) it was proved the
existence of solitons and compactons (solitons with compact
support).\par In this work we shall examine the local well-posedness
of the Cauchy problem \eqref{eq:rBOZK} in the anisotropic Sobolev
spaces
\begin{equation}
  \label{eq:1.1}
H^{s_1,s_2}(\re^2)=\{f \in \SSh' \mid
(1+|\xi|^{s_1}+|\eta|^{s_2})\widehat{f} \in L^2(\re^2)\},  
\end{equation}
for $s_1$ and $s_2$ positive real numbers such that
$\frac1{s_1}+\frac1{s_2}<2$, and the weighted Sobolev spaces
\begin{equation}\label{eq:1.2}
\mathcal{F}_{r_1,r_2}^{s_1,s_2}(\re^2) = H^{s_1,s_2}(\re^2) \cap L^2_{r_1,r_2}(\re^2),
\end{equation}
where
\begin{equation}\label{eq:1.3}
L^2_{r_1,r_2}=\left\{f \in L^2(\re^2) \mid
(|x|^{r_1}+ |y|^{r_2})f \in L^2(\re^2)\right\},
\end{equation}
for $s_2\ge r_2$, $s_2\ge 2r_1$ and $r_1<5/2$.  We shall show,
further, that the unique solution to \eqref{eq:rBOZK} with a
particular condition on initial value and belonging to $\mathcal{ F}_{
  r_1, r_2}^{s_1,s_2}(\re^2)$, for $r_1\ge 5/2$, in two different
times is identically $0$. To do this we follow ideas presented in
\cite{FonsecaGui}, \cite{Fonseca1} and \cite{Fonseca2}.\par The plan
in this paper is the following. In Section 2 we present the
preliminaries results and notations that we will use in the rest of
this work. In section 3, in a first subsection we examine the local
well-posedness of \eqref{eq:rBOZK}, in a second subsection we examine
the unique continuation of its solutions.
\section{Preliminary results}
In this paper we systematically use the following notations.
\begin{enumerate}
\item $\SSh({\re}^n)$  is the Schwartz space. If $n=2$,
  we simply write $\SSh$.
\item $\SSh'({\re}^n)$ is the tempered  distributions.
  If $n=2$, we simply write $\SSh'$.
\item For $f\in{\SSh}'({\re}^n)$, $\widehat{f}$ is the Fourier
  transform of $f$ and $\check{f}$ is inverse Fourier transform
  of $f$. We recall that  $$ \widehat f(\xi) = {(2\pi)^{-\frac n2}}
  \int_{\re^n} f(x)e^{i\langle x,\xi\rangle} dx, $$ for all $\xi\in
  \re^n$, when $f\in \SSh(\re^n)$.
\item $\mathscr H$ is the Hilbert transform (in $x$). If
  $f\in{\SSh}({\re})$,  $$\mathscr H f (x) = \sqrt\frac2\pi \left(\mathrm{v.p.}
  \int_{-\infty}^\infty \frac1{y-x} f(y)\, dy\right). $$
\item For $s\in{\re}$,
$H^s=H^s({\re}^2)$
is the Sobolev space of order $s$.  
\item The inner product in $H^s$ is $\langle f,g\rangle_s=\int_{{\re}^2}(1+\xi^2+\eta^2)^s\widehat{f}\overline{\widehat{g}}d\xi{d}\eta.$
\item For $s_1$ and $s_2\in{\re}$ we denote $H^{s_1,s_2}=H^{s_1,
    s_2}({\re}^2)$ the anisotropic Sobolev space defined as in the
  introduction. 
\item $\Lambda^s=(1-\Delta)^{s/2}$.
\item $L_{p,s}(\re^n)=\{ f\in \SSh'(\re^n) \,\big|\, \Lambda^sf\in L_p(\re^n)\}$. 
\item For $f \in L_{p,s}(\re^2)$, $|f|_{p,s}= \|\Lambda^sf\|_{L_p(\re^2)}$.
\item $[A,B]$ denotes the commutator of $A$ and $B$.
\end{enumerate}
 
In this section we present some preliminaries concepts and results
that we will use as important inputs in this work. Let us see.
\subsection{Anisotropic weighted Sobolev spaces}
In most of this work we move within the context of anisotropic Sobolev
spaces $H^{s_1,s_2}(\re^2)$ and anisotropic weighted Sobolev spaces
$\mathcal F_{r_1,r_2}^{s_1,s_2}(\re^2)$ (see \eqref{eq:1.1} and
\eqref{eq:1.2} for their definition). In these spaces we have the
following version of the Sobolev lemma.
\begin{lema}[Sobolev]\label{lemsob}
  Let $s_1$ and $s_2$ be positive real numbers such that $\frac1{s_1}
  + \frac1{s_2}< 2$. Then, $H^{s_1,s_2}(\re^2)\subset
  C_\infty(\re^2)$ (the set of continuous functions in $\re^2$
  vanishing at infinity), with continuous inclusion.
\end{lema}
\begin{proof}
  In view of $$ \widehat f=\left( (1+|\xi|^{s_1} + |\eta|^{s_2} \right)^{-1}
  \left( (1+|\xi|^{s_1} + |\eta|^{s_2} \right) \widehat f,$$ it is
  enough to see that $\left(1+|\xi|^{s_1} + |\eta|^{s_2} \right)^{-1} \in
  L^2(\re^2)$, or, in other words, $\left(1+|\xi|^{s_1} +
    |\eta|^{s_2} \right)^{-2} $ is integrable. Now, since 
 $$\left(1+|\xi|^{s_1} + |\eta|^{s_2} \right)^{-1}\simeq \left(
    1+|\xi|+ |\eta|^{s_2/s_1} \right)^{-2s_1}$$ and $$
  \int_{-\infty}^\infty \left( 1+|\xi|+ |\eta|^{s_2/s_1}
  \right)^{-2s_1}\,d\xi = \frac2{2s_1-1} \left( 1+ |\eta|^{s_2/s_1}
  \right)^{-2s_1+1},$$ from Tonelli theorem we have that $\left(
    1+|\xi|^{s_1} + |\eta|^{s_2} \right)^{-1}$ is integrable if
  (and only if) $$\frac1{s_1} + \frac1{s_2}< 2.$$
\end{proof}
Before considering other properties of these spaces, we enunciate the
following two results about commutators of operators that are part of the
important stock of tools used in mathematical analysis.\par
\begin{prop}[Kato-Ponce inequality]\label{deskatoponce}
  Let $s>{0}$, $1<p<\infty$, $\Lambda=(1-\Delta^2)^{1/2}$ and $M_f$
  the multiplication operator by $f$. Then,
\begin{equation}\label{lamilagrosa2}
\left|[\Lambda^{ s}, M_f] g\right |_{p}\le c\left( |\nabla
f|_{\infty} |\Lambda^{ s-1}g|_{p}+ |\Lambda^{ s}f|_{p}| g|_{\infty}
\right),
\end{equation}
for all $f$ and $g\in \SSh(\re^n)$
\end{prop}
\begin{corol}\label{needglob}
For $f$ and $g\in \SSh(\re^n)$,
$$|fg|_{s,p}\le c\left( |f|_{\infty} |\Lambda^{ s}g|_{p}+ |\Lambda^{
    s}f|_{p}| g|_{\infty} \right).
$$
\end{corol}
Also we need the following lemma.
\begin{prop}\label{conmutador}
If $f \in H^{1/2}(\re)$ and $\rho \in C_0^{\infty}(\re)$, then
\begin{equation}
\|[D_x^{1/2},\rho ]f\|_{L^2(\re)} \leq C \| \widehat{D_x^{1/2}\rho }\|_{L^1(\re)} \|f \|_{L^2(\re)}.
\end{equation}
\end{prop}
\begin{proof}
Let $f \in H^{1/2}(\re)$. then,
\begin{align*}
|([D^{1/2},\rho]f)^{\wedge}(\xi)| & = \left| \int_{\re} (|\xi|^{1/2}-|\eta|^{1/2}) |\widehat{\rho}(\xi-\eta)\widehat{f}(\eta)| \, d\eta \right|\\
& \leq   \int_{\re} ||\xi-\eta|^{1/2}\widehat{\rho}(\xi-\eta)| |\widehat{f}(\eta)| \, d \eta .
\end{align*}
Thus, from Young inequality and Plancherel theorem it follows the
proposition.
\end{proof}
\begin{corol}\label{anisotropic}
  Let $s_1$ and $s_2$ be such that $\frac1{s_1} + \frac1{s_2}< 2$.
  Then, the space $H^{s_1,s_2}(\re^2)$, with the punctual product of
  functions, is a Banach algebra. In particular,
  \[\|fg\|_{s_1,s_2} \leq c \|f\|_{s_1,s_2}\|g\|_{s_1,s_2}.\]
\end{corol}
\begin{proof}
  Observe that if $f$ and $g\in \SSh$, from Corollary \ref{needglob},
  we have, for any $y\in \re$, 
  \begin{multline*}
  \|(1- \partial_x^2)^{s_1/2}(fg)(\cdot,y) \|_{L^2(\re)} \le\\ \le \|f\|_\infty
  \|(1- \partial_x^2)^{s_1/2}(g)(\cdot,y) \|_{L^2(\re)} +
  \|(1- \partial_x^2)^{s_1/2}(f)(\cdot,y) \|_{L^2(\re)} \|g\|_\infty.
  \end{multline*}  
Then,
\begin{multline*}
  \|(1- \partial_x^2)^{s_1/2}(fg) \|_{L^2(\re^2)} \le\\ \le \|f\|_\infty
  \|(1- \partial_x^2)^{s_1/2}(g) \|_{L^2(\re^2)} +
  \|(1- \partial_x^2)^{s_1/2}(f) \|_{L^2(\re^2)} \|g\|_\infty.
  \end{multline*}
Analogously,
\begin{multline*}
  \|(1- \partial_y^2)^{s_2/2}(fg) \|_{L^2(\re^2)} \le\\ \le \|f\|_\infty
  \|(1- \partial_y^2)^{s_2/2}(g) \|_{L^2(\re^2)} +
  \|(1- \partial_y^2)^{s_2/2}(f) \|_{L^2(\re^2)} \|g\|_\infty.
  \end{multline*}
From these two inequalities and Sobolev lemma (Lemma \ref{lemsob}) it
follows the corollary.
\end{proof}
\begin{corol}\label{algban}
  Under the same assumptions of immediately previous corollary, for any
   $r_1$ and $r_2$  positive real numbers,
  $\mathcal{F}_{r_1,r_2}^{s_1,s_2}$ is a Banach algebra.
\end{corol}
\begin{proof}
Let $f$ and $g$ in $\mathcal{F}_{r_1,r_2}^{s_1,s_2}$, then
\begin{align*}
\|fg\|_{\mathcal{F}_{r_1,r_2}^{s_1,s_2}}^2 & = \|fg\|_{s_1,s_2}^2+ \|fg\|_{L_{r_1,r_2}^2}^2\\
& \leq C\|f\|_{s_1,s_2}^2 \|g\|_{s_1,s_2}^2 + \|f\|_{L^{\infty}}^2 \|g\|_{L_{r_1,r_2}^2}^2\\
& \leq  C\|f\|_{s_1,s_2}^2 \|g\|_{\mathcal{F}_{r_1,r_2}^{s_1,s_2}}^2 \\
& \leq C\|f\|_{\mathcal{F}_{r_1,r_2}^{s_1,s_2}}^2 \|g\|_{\mathcal{F}_{r_1,r_2}^{s_1,s_2}}^2
\end{align*} 
\end{proof}
\begin{defi}
  Let $\omega$ be a non-negative locally integrable function on $\re$.
  We say that $\omega$ \emph{satisfies the  $A_p$ condition}, for
  $1<p<\infty$, if there exists a positive real number $C$ such that
\begin{equation}\label{Ap}
 \left ( \frac{1}{|I|}\int_{I} \omega \right )\left ( \frac{1}{|I|}
   \int_{I} \omega^{-\frac1{p-1}} \right )^{p-1} \le C
\end{equation}
for all nonempty open interval $I$ in $\re$.
\end{defi}
\begin{ejem}\label{Ap2}
An immediate example of a function that satisfies $A_p$ condition is
$w(x)=(\gamma+|x|^\alpha)^r$, for $\gamma\ge 0$ and $-1<r\alpha<p-1$.
\end{ejem}  
A very interesting fact about the $A_p$ condition is the following theorem.
\begin{teo}\label{condicionAp}
  $\omega$ satisfies the $A_p$ condition if, and only if,  Hilbert
  transform is a bounded operator on $L^{p}(\omega(x)dx)$. In other
  words, $\omega$ satisfies the  $A_p$ condition if, and only if,
\begin{equation}
\left ( \int_{-\infty}^{\infty}|\mathscr{H}f|^{p} \omega(x)\, dx\right )^{\frac{1}{p}} \leq c^{*}\left ( \int_{-\infty}^{\infty}|f|^{p} \omega(x) \, dx \right )^{\frac{1}{p}}
\end{equation}
for all $f\in L^{p}(\omega(x)dx)$.
\end{teo}

Key ingredients in this work is given by the properties of the
homogeneous fractional derivative and the Stein derivative (this last
very related to the first). We recall that the fractional derivative is
defined by
 \begin{equation}
D^bf = (2\pi |\xi|^b \widehat{f}\ )^{\vee}.
\end{equation}
for each tempered distribution $f$ such that $|\xi|^b \widehat{f}$ is
also a tempered distribution.  If $0<b<1$ and $1<p< \infty$ we have that
\begin{equation}\label{eq:2.4}
\|D^b(fg)\|_p \leq C( \|g\|_{\infty} \|D^bf\|_p+\|f\|_{\infty}\|D^bg\|_p)
\end{equation}
(for a proof of this last inequality see \cite{katoponce}). Furthermore,
we have the following result proved in \cite{KPV1993}
\begin{teo}\label{teo:leib}
If $0<b<1$, then 
\begin{enumerate}[(i)]
\item For $1<p<\infty$
\begin{equation}\label{bu}
\|D^b(fg)-fD^bg-gD^bf\|_p \leq c\|f\|_{\infty}\|D^bg\|_p
\end{equation}

\item For $1\leq p<\infty$ , $b_1, \, b_2 \, \in \, [0,b]$ with $b_1+b_2=b$ and $p_1, \, p_2 \, \in \, (1,\infty)$ with $\frac{1}{p_1}+\frac{1}{p_2}=\frac1p$
\begin{equation}
\|D^b(fg)-fD^bg-gD^bf\|_p \leq c\|D^{b_1}f\|_{p_1}\|D^{b_2}g\|_{p_2}
\end{equation}

\end{enumerate}
\end{teo}
In particular, for $1<p<\infty$, from \eqref{bu} we have the following
inequality that improves \eqref{eq:2.4}
\begin{equation}\label{eq:2.8}
\|D^b(fg)\|_p \leq C(\|gD^bf\|_p +\|f\|_{\infty}\|D^bg\|_p)
\end{equation}
\begin{corol}\label{acotop}
  The operator $B = - \partial_x (1+\mathscr{H}\partial_x)^{-1} $ is
  bounded on $\mathcal{F}_{r_2,r_2}^{s_1,s_2}$, for
  $ r_1< 5/2$.
\end{corol}
\begin{proof} It is clear that 
\begin{align*}
  \|B(\varphi)\|_{{s_1,s_2}} &\le \|\varphi\|_{{s_1,s_2}} \intertext
  {and} \|(1+y^2)^{ \frac {r_2}2} B(\varphi)\|_{L^2} &=\|B((1+y^2)^{
    \frac {r_2}2} \varphi)\|_{L^2}\le \| (1+y^2)^{ \frac {r_2}2}\varphi\|_{L^{2}}.
\end{align*}
Let us examine what happens with $\||x|^{ r_2} B(\varphi)\|_{L^2} $. Since
$$\partial_\xi \left( \frac \xi{1+|\xi|} \right) =\frac {1}
{(1+ |\xi|)^2} \quad\text y \quad \partial_\xi^2 \left( \frac \xi{1+|\xi|} \right) =-\frac {2\sgn \xi}
{(1+ |\xi|)^3},$$ we have $$\|x^2 B\varphi\|_{L^2} =
\left\| \partial_\xi^2 \left( \frac \xi{1+|\xi|}  \widehat \varphi \right)
\right\|_{L^2} \le c\|(1+x^2) \varphi\|_{L^2}.$$
A simple interpolation argument shows that
\begin{equation}\label{eq:2.9}
\|(1+x^2)^{\frac {r_1}2} B\varphi\|_{L^2}\le C \|(1+x^2)^{\frac
  {r_1}2}  \varphi\|_{L^2},
\end{equation}
for $ r_1\le 2$.
Now suposse that $2< r_1<5/2$. In this case $r_1=2+b$, for
$0<b<1/2$. Therefore,
\begin{equation}\label{eq:2.10}
\begin{aligned}
\||x|^{ r_1}B\varphi \|=& \left \|D_\xi^b \partial_\xi^2 \left( \frac
    \xi{1+|\xi|}  \widehat \varphi \right) \right\|_{L^2} \\ \le &C \bigg(  \left \|
    D_\xi^b  \left( \frac \xi{1+|\xi|}  \partial_\xi^2 \widehat \varphi
    \right) \right\|_{L^2} + \left\| 
    D_\xi^b  \left( \frac 1{(1+|\xi|)^2}  \partial_\xi \widehat \varphi
    \right) \right\|_{L^2} \\ & + \left\|
    D_\xi^b  \left( \frac {-2\sgn \xi} {(1+|\xi|)^3}  \widehat \varphi
    \right) \right\|_{L^2}\bigg).
\end{aligned}
\end{equation}
Now, examining each of the terms on the right side of the above
inequality we have, first, as we have already proved that $B$ is a
bounded operator in $L^2((1+x^2)^b \, dxdy)$, for $0<b<1/2$, $$\left
  \| D_\xi^b \left( \frac \xi{1+|\xi|} \partial_\xi^2 \widehat \varphi
  \right) \right\|_{L^2} = \||x|^b B(x^2\phi) \|_{L^2} \le
C\|(1+x^2)^{\frac {r_1}2}\phi \|_{L^2}. $$ Second, from inequality
\eqref{eq:2.8} and since $\frac 1{(1+|\xi|)^2}\in H^1(\re)$,
\begin{align*} \left\| 
    D_\xi^b  \left( \frac 1{(1+|\xi|)^2}  \partial_\xi \widehat \varphi
    \right) \right\|_{L^2} \le & c\left( \left\| 
    D_\xi^b  \left( \frac 1{(1+|\xi|)^2}\right)\right\|_\infty
  \| \partial_\xi \widehat \varphi  \|_{L^2} + \|D_\xi^{b+1}
\widehat \varphi\|_{L^2} \right)\\ \le &c\|(1+x^2)^{\frac {r_1}2}\phi \|_{L^2},
\end{align*}
for $0<b<1/2$. Finally, from inequality \eqref{eq:2.8},
Theorem \ref{condicionAp} and since $\frac 1{(1+|\xi|)^3}\in
H^1(\re)$, we have \begin{align*}\left\| D_\xi^b \left( \frac {-2\sgn \xi}
      {(1+|\xi|)^3} \widehat \varphi \right) \right\|_{L^2} =&2\left
    \||x|^b \mathscr H \left( \frac {1} {(1+|\xi|)^3} \widehat \varphi
    \right)^\vee \right\|_{L^2}\\ \le & \left\||x|^b \left( \frac {1}
      {(1+|\xi|)^3} \widehat \varphi \right)^\vee \right\|_{L^2} =
  \left\| D_\xi^b \left( \frac {1} {(1+|\xi|)^3} \widehat \varphi
    \right) \right\|_{L^2}\\ \le &c\left( \left\| D_\xi^b
      \left( \frac {1} {(1+|\xi|)^3} \right) \right\|_\infty \| \widehat
    \varphi \|_{L^2} +\left\| \frac {1} {(1+|\xi|)^3}  D_\xi^b
      \widehat \varphi \right\|_{L^2}\right)\\ \le &c\|(1+x^2)^{\frac
    {r_1}2}\varphi \|_{L^2},
\end{align*}
for $0<b<1/2$. Then, from these last  estimates and
\eqref{eq:2.10} it follows \eqref{eq:2.9}, for $2< r_1<5/2$. It proves 
the corollary. 
\end{proof}
We recall that Stein derivative is defined by: for $b \in
(0,1)$ and complex-valued measurable function $f$ on $\re^n$, 
\begin{equation}
\mathcal{D}^{b}f(x)= \left(\int_{\re^{n}}
  \frac{|f(x)-f(y)|^2}{|x-y|^{n+2b}} \, dy \right)^{\frac12}. 
\end{equation}
\begin{defi}
For $s\in \re$ we denote by $L_s^p(\re^n)$ the space of all 
functions $f$ in $L^p(\re^n)$ such that $(1-\Delta)^sf\in L^p(\re^n)$.
The norm in this space is defined by $$\|f\|_{s,p}=\|
(1-\Delta)^{s/2}f\|_p,$$ for all $f$ in $L^p(\re^n)$. 
\end{defi}
In the foollowing theorem it is given a characterization of the spaces
$L_s^p(\re^n)$ in Stein derivative terms.
\begin{teo}
Suposse $b \in (0,1)$ and ${2n}/ {n+2b} \leq p < \infty$. Then, $f \in L_b^{p}(\re^n)$ if, and only if,
\begin{enumerate}[(a)]
\item $f \in L^{p}(\re^n)$
\item $\displaystyle{\mathcal{D}^{b}f(x) \in L^{p}(\re^{n})}$. 
\end{enumerate}
Furthermore,
\begin{equation*}
\|f\|_{b,p}= \|(1- \Delta)^{b/2}f\|_p \cong \|f\|_p+\|D^b f\|_p \cong \|f\|_p+\|\mathcal{D}^b f\|_p
\end{equation*}
\end{teo}
\begin{proof}
See \cite{stein1961} or \cite{stein1970}.
\end{proof}
The next theorem is analogous to Theorem \ref{teo:leib} for the Stein
derivative case.
\begin{teo}
  Let $b \in (0,1)$ and $1 \leq p < \infty$. If $f, g: \re^n \to
  \mathbb{C}$ are measurable functions, then
\begin{align}
\text{ a)}\quad&\mathcal{D}^b(fg)(x) \leq \|f\|_{\infty} \mathcal{D}^bg(x)+ |g(x)|\mathcal{D}^bf(x)\\
\text{ b)}\quad&\|\mathcal{D}^b(fg)\|_p \leq \|f\|_{\infty} \|\mathcal{D}^bg\|_p + \|g\mathcal{D}^bf\|_p\\
\text{ c)}\quad&\|\mathcal{D}^b(fg)\|_2 \leq \|f\mathcal{D}^bg\|_2 +
\|g\mathcal{D}^bf\|_2 \label{equ1}
\end{align}
\begin{proof}
See \cite{napo}
\end{proof}
\end{teo}

\begin{lema}\label{interpolacion}
Let $a$ and $b$ positive real numbers. If $J_u =
(1-\partial_u^2)^{1/2}$ and $\langle v\rangle=(1+v^2)^{1/2}$, then,
for any $\theta \in (0,1)$,

\[\text{ a)}\quad \|J_u^{\theta a}( \langle v\rangle^{(1-\theta)b}f)\| \leq c \|\langle v \rangle ^b f\|^{1- \theta} \|J_u^af\|^{\theta} \]
\[\text{ b)} \quad \|\langle v\rangle^{(1-\theta)b}J_u^{\theta a}f)\| \leq c \|\langle v \rangle ^b f\|^{1- \theta} \|J_u^af\|^{\theta}, \]
for $f\in L^2(\re^2)$, where $u$ and $v$ can be taken as $x$ and $y$
interchangeably.
\end{lema}
\begin{proof}
  The proof is exactly the same of the Lemma 1 in
  \cite{Fonseca1} (See also \cite{napo}, Lemma 4); indeed, it is very
 simpler when $u$ and $v$ are different.
\end{proof}

The next theorem is used in the proof of the unique continuation of
solutions of the equation \eqref{eq:rBOZK}.

\begin{teo}\label{lim}
  Let $p \in (1, \infty)$ and $f \in L^p(\re)$. If for some $x_0 \in
  \re$ such that $f(x_0+)$ and $f(x_0-)$ there exist and are different,
  then, for any $\delta >0$, $D^{1/p}f \notin
  L_{loc}^p(B(x_0,\delta))$. In particular, $f \notin
  L_{\frac{1}{p}}^p(\re)$
\end{teo}
\begin{proof}
  Observe that $D^{1/p}f(x)\sim \dfrac1{|x-x_0|^{1/p}}$, when $x\to x_0$. 
\end{proof}
\section{Well-posedness}
\subsection{The linear problem in weighted spaces}
In the next two lemmas, it will be given explicitly the formulae for
derivatives of the symbol (via Fourier transform) of the unitary group
that generates the solutions to the linear equation associated to
\eqref{eq:rBOZK}. 
 \begin{lema}\label{derivadas1}
 Let  $F(\xi,\eta,t)=e^{b(\xi, \eta)t}$, where
 $b(\xi,\eta)=\frac{i\xi\eta^2}{1+|\xi|} $. Then,
 \begin{align*}\partial _{\xi} F(t,\xi,\eta)= (it)(1+|\xi|)^{-2}\eta^2
    F(t,\xi, \eta)\end{align*}
\begin{align*} 
\partial^2 _{\xi} F(t,\xi,\eta) =& -2it(1+|\xi|)^{-3} \sgn(\xi)\eta^2
F(t, \xi,\eta) + (it)^2(1+|\xi|)^{-4} \eta^4 F(t,\xi,\eta)\\ 
\partial^3 _{\xi} F(t,\xi,\eta) =& 4(it)\delta \eta^2+
6it(1+|\xi|)^{-4}\eta^2 F(t, \xi,\eta)- \\&-6(it)^2(1+|\xi|)^{-5}
\sgn(\xi)\eta^4 F(t,\xi,\eta)+ (it)^3(1+|\xi|)^{-6} \eta^6 F(t,
\xi,\eta) \\
\partial^4 _{\xi} F(t,\xi,\eta) =& 4(it)\delta' \eta^2 -6(it)^2\delta
\eta^4-24(it)(1+|\xi|)^{-5}\sgn(\xi)\eta^2 F(t, \xi,\eta)+ \\
&+36(it)^2(1+|\xi|)^{-6} \eta^4F(t,\xi,\eta)-\\ &-12(it)^3 (1+|\xi|)^{-7}
\sgn(\xi) \eta^6F(t, \xi,\eta) + (it)^4 (1+|\xi|)^{-8}\eta^8 F(t,\xi, \eta)
\intertext{In general,  for $j \geq 5$}
\partial_{\xi}^{j}F(t,\xi,\eta) =& 4(it)\delta^{(j-3)}\eta^2 -6(it)^2
\delta^{(j-4)}\eta^4 -\\
&- \sum \limits_{k=1}^{j-4} (a_k (it)^{k} \eta^{2k}+b_k
(it)^{k+2}\eta^{2(k+2)}) \delta^{(k-1)} +\\
&+\sum \limits_{k=1}^{j} a_k (it)^{j}(1+|\xi|)^{-(j+k)}
(\sgn(\xi))^{j-k} \eta^{2k}F(t, \xi, \eta)
\end{align*}
where $\delta=\delta_{\xi=0}$ is  the Dirac's delta measure concentred
in the straight line $\xi=0$, and $a_k$ and $b_k$ are constants
dependending only on $k$.
\end{lema}
\begin{corol}\label{pesoenx}
  Let $A=\partial_x(1+\mathscr{H}\partial_x)^{-1} \partial_{yy}$ and
  $E(t)=e^{At}$.  For $s_1$, $s_2 \in \re$ and $r \in \mathbb{N}$ such
  that $s_2 \geq 2r$, we have:
\begin{enumerate}
\item If $r=1$ or $2$, 
  \[\|E(t)\varphi\|_{\mathcal{F}_{r,0}^{s_1,s_2}} \leq
  P_r(t)\|\varphi\|_{\mathcal{F}_{r,0}^{s_1,s_2}},\] where $P_r(t)$ is
  a polynomial of degree $r$ in $t$ with  positive coefficients.
\item If $r\geq 3$ and $\varphi \in \mathcal{F}_{r,0}^{s_1,s_2}$,
  $E(t)\varphi \in C([0,\infty) ; \mathcal{F}_{r,0}^{s_1,s_2})$ if, and only
  if,
\[\partial_{\xi}^j \widehat{\varphi}(0,\eta)=0,\text{ for all $\eta$
y } j=0,1,2 \dots r-3\]
\end{enumerate}
\end{corol}
\begin{proof}
For $r=1$, from the above lemma we have,
\begin{align*}
\|E(t)\varphi\|_{\mathcal{F}_{1,0}^{s_1,s_2}}^2 & = \|E(t)\varphi\|_{s_1,s_2}^2+ \|xE(t)\varphi\|_{L^2}^2\\
&\leq  \|\varphi\|_{s_1,s_2}^2 + \int_{\re^2} |xE(t)\varphi|^2 \, dxdy \\
& = \|\varphi\|_{s_1,s_2}^2 + \int_{\re^2} |\partial_{\xi}(F(t,\xi,\eta)\widehat{\varphi})|^2 \, d\xi d\eta \\
&\leq   \|\varphi\|_{s_1,s_2}^2 + \int_{\re^2} |(\partial_{\xi}F(t,\xi,\eta))\widehat{\varphi}|^2 +|F(t,\xi,\eta) \partial_{\xi} \widehat{\varphi}|^2\, d\xi d\eta\\
&\leq   \|\varphi\|_{s_1,s_2}^2 + \int_{\re^2} |\frac{it\eta^2}{(1+|\xi|)^2}F(t,\xi,\eta))\widehat{\varphi}|^2 + |\partial_{\xi} \widehat{\varphi}|^2\, d\xi d\eta \\
&\leq  (1+t^2) \|\varphi\|_{s_1,s_2}^2+\|\varphi\|_{L_{1,0}^2}^2\\
& \leq P_{1}^2(t) \|\varphi\|_{\mathcal{F}_{1,0}^{s_1,s_2}}^2
\end{align*}
If $r=2$
\begin{align*}
  \|E(t)\varphi\|_{\mathcal{F}_{2,0}^{s_1,s_2}}^2  =& \|E(t)\varphi\|_{s_1,s_2}^2+ \|x^2E(t)\varphi\|_{L^2}^2\\
  =  &\|\varphi\|_{s_1,s_2}^2 + \int_{\re^2} |x^2E(t)\varphi|^2 \, dxdy \\
   = &\|\varphi\|_{s_1,s_2}^2 + \int_{\re^2}
  |\partial_{\xi}^2(F(t,\xi,\eta)\widehat{\varphi})|^2 \, d\xi d\eta\\
  \phantom{\|E(t)\varphi\|_{\mathcal{F}_{2,0}^{s_1,s_2}}^2} \leq&
  \|\varphi\|_{s_1,s_2}^2 + \int_{\re^2}
  \left(|(\partial_{\xi}^2F(t,\xi,\eta))\widehat{\varphi}|^2 +
    |(\partial_{\xi}F(t,
    \xi,\eta))\partial_{\xi}\widehat{\varphi}|^2\right) d\xi d\eta + \\
  &+\int_{\re^2} |F(t,\xi,\eta) \partial_{\xi}^2 \widehat{\varphi}|^2
  \, d\xi d\eta  \\
  \leq  &\|\varphi\|_{s_1,s_2}^2 + \int_{\re^2} \left | \left
      (\frac{2it\eta^2 \sgn(\xi)}{(1+|\xi|)^3}+ \frac{(it)^2
        \eta^4}{(1+|\xi|)^4}\right)F(t,\xi,\eta ) \widehat{\varphi}
  \right |^2 \, d\xi d\eta + \\
  &+ \int_{\re^2}\left |\frac{2it
      \eta^2}{(1+|\xi|)^2}F(t,\xi,\eta)\partial_{\xi}\widehat{\varphi}\right
  |^2 \, d\xi d\eta + \int_{\re^2} |\partial_{\xi}^2\varphi|^2 \, d\xi d\eta\\
  \leq  &\|\varphi\|_{s_1,s_2}^2 + 4t^2 \|\partial_y\varphi\|_{L^2}^2+
  t^4\|\partial_y^4 \varphi\|_{L^2}^2+ 4t^2 \|x\partial_y^2 \varphi\|_{L^2}^2 +\|x^2 \varphi\|_{L^2}^2 
\end{align*}
Now, since from Lemma  \ref{interpolacion}
\begin{equation}\label{eq:2.13}
\begin{aligned}
\|x \partial_y^2 \varphi\|^2 
& \leq \|\langle x \rangle J^2_y\varphi\|^2 \\
& \leq c\|\langle x \rangle^2 \varphi\|\|J_y^4\varphi\|\\
& \leq c( \|\varphi\|_{s_1,s_2}^2+ \|\varphi\|_{L_{2,0}^2}^2),
\end{aligned}
\end{equation}
we have that
$$
\|E(t)\varphi\|_{\mathcal{F}_{2,0}^{s_1,s_2}}^2   \leq P_{2}^2(t)
\|\varphi\|_{\mathcal{F}_{2,0}^{s_1,s_2}}^2.
$$

Assume that $r=3$ and $\widehat\varphi(0, \eta)=0$. Inasmuch as $t \delta F
\widehat{\varphi}=t \delta F(t,0,\eta) \widehat{\varphi}(0, \eta)=0$, then 
\begin{align*}
\|E(t)\varphi  \|_{\mathcal{F}_{3,0}^{s_1,s_2}}^2  =&
\|E(t)\varphi\|_{s_1,s_2}^2+ \|x^3E(t)\varphi\|_{ L^2}^2\\
\leq & \|\varphi\|_{s_1,s_2}^2 + \int_{\re^2} |x^3E(t)\varphi|^2 \, dxdy 2\\
 \le &\|\varphi\|_{s_1,s_2}^2 + \int_{\re^2} |\partial_{\xi}^3(F( t,
 \xi, \eta)\widehat{\varphi})|^2 \, d\xi d\eta \\
\leq  & \|\varphi\|_{s_1,s_2}^2 + \int_{\re^2} |(\partial_{ \xi}^3F( t, \xi,
\eta)) \widehat{\varphi}|^2 d\xi d\eta + |3(\partial_{\xi}^2F( t, \xi,
\eta))\partial_{\xi}\widehat{\varphi}|^2 d\xi d\eta +\\
&+\int_{\re^2} |3\partial_{\xi}F(t,\xi,\eta) \partial_{\xi}^2
\widehat{ \varphi}|^2\, d\xi d\eta + \int_{\re^2} |F(t, \xi,
\eta) \partial_{ \xi}^3 \widehat{\varphi}|^2\, d\xi d\eta \\
\leq  & \|\varphi\|_{s_1,s_2}^2 +\\ &+ \int_{\re^2} \left | \left
    (\frac{6it \eta^2}{(1+|\xi|)^4}+ \frac{6 (it)^2\eta^4}{
      (1+|\xi|)^5 }+\frac{(it)^3 \eta^6}{(1+|\xi|)^6}\right)F(t, \xi,
  \eta) \widehat{ \varphi} \right |^2  d\xi d\eta + \\ 
&+9\int_{\re^2} \left | \left (\frac{2it\eta^2 \sgn(\xi)}{(1+
      |\xi|)^3} +\frac{(it)^2 \eta^4}{(1+ |\xi|)^4} \right)F( t, \xi,
  \eta) \partial_{ \xi} \widehat{\varphi} \right |^2 \, d\xi d\eta \\
&+ 9\int_{\re^2}\left |\frac{2it \eta^2}{(1+|\xi|)^2}F(t,\xi,
  \eta) \partial_{\xi}^2 \widehat{\varphi}\right |^2\, d\xi d\eta +
\int_{\re^2} |\partial_{ \xi}^3\varphi|^2 \, d\xi d\eta\\ \leq & \|
\varphi\|_{ s_1,s_2}^2 + 36t^2\|\partial_y^2 \varphi\|_{L^2}^2+
36t^4\|\partial_y^4  \varphi\|_{L^2}^2+9 t^6\|\partial_y^6 \varphi\|_{
  L^2}^2 +\\&+36t^2\|x \partial_y ^2\varphi\|_{L^2}^2 +
9t^4\|x \partial_y^4 \varphi\|_{L^2}^2+ 36t^2\|x^2 \partial_y^2
\varphi\|_{ L^2}^2 + \|x^3 \varphi\|_{L^2}^2  
\end{align*}
Now, arguing as in \eqref{eq:2.13}, we have
\begin{align*}
\|x \partial_y^4 \varphi\|  
& \leq c ( \|\langle x \rangle^3 \varphi\| + \|J_y^6 \varphi\| ),\\
\|x^2 \partial_y^2 \varphi\|  
& \leq c ( \|\langle x \rangle^3 \varphi\| + \|J_y^6 \varphi\| ). \\
\end{align*}
Therefore,
$$
\|E(t)\varphi\|_{\mathcal{F}_{3,0}^{s_1,s_2}}^2   \leq P_{3}^2(t)
\|\varphi\|_{\mathcal{F}_{3,0}^{s_1,s_2}}^2.
$$
Let us suposse that $E(t)\varphi \in C([0,\infty) ; \mathcal{F }_{r,0}
^{s_1, s_2})$. Since $$\partial_\xi ^3 (F\varphi)= \sum_{k=0}^3
{3\choose k} \partial_\xi^{k} F \partial_\xi^{3-k}\widehat \varphi, $$
employing the same arguments we used above in this proof, we can see
$\partial_\xi^{k} F \partial_\xi^{3-k}\widehat \varphi\in L^2(\re^2)$,
for $k=0,1$ and $2$, and $\partial_\xi^{3} F \widehat \varphi -4(it)
\eta^2 \widehat \varphi(0,\eta)\delta \in L^2(\re^2)$. Then, it
follows that $4(it) \eta^2 \widehat \varphi(0,\eta)\delta\in
L^2(\re^2)$, which shows that $\widehat \varphi(0,\eta)=0$, for all
$\eta$.\par The proof for $r\ge 4$ is basically the same for $r=3$
with the help of an induction argument.
\end{proof}
Now, let us see the derivatives of the group with respect to the
variable $\eta$.
\begin{lema}\label{derivadas2}
  Let $F(t,\xi,\eta)=e^{b(\xi, \eta)t}$, where $b(\xi,\eta )=\frac{
    i\xi \eta^2}{1+ |\xi|}$, then, for $j \geq 1$,
\begin{align*}
  \partial_{\eta}^{2j}F(t, \xi, \eta)&= \sum
  \limits_{k=0}^{j}a_{k}\eta^{2k} \left( \frac{2i t \xi}{1+ |\xi|}
  \right)^{j+k}F(t,\xi,\eta) \\
  \intertext{and}
 \partial_{\eta}^{2j+1}F(t,\xi,\eta)&= \sum \limits_{k=0}^{j}
 a_{k}\eta^{2k+1} \left( \frac{2i t \xi}{1+|\xi|}
 \right)^{j+1+k}F(t,\xi,\eta)
\end{align*}
\end{lema}

Now we have the following corollary analogous to Corollary
\ref{pesoenx},  however unlike this one it is not requered the
restrictive condition in item 2.
\begin{corol}\label{pesoeny}
  Let $E$ be as in Corollary \ref{pesoenx}. Then, for $s_1,\
  s_2 \in \re$ and $r \in \mathbb{N}$ with $s_2 \geq r$, we have
\begin{equation}
  \|E(t)\varphi\|_{\mathcal{F}_{0,r}^{s_1,s_2}} \leq P_r(t)\|\varphi\|_{\mathcal{F}_{0,r}^{s_1,s_2}},
\end{equation}
where $P_r(t)$ is a polynomial of degree $r$ with positives  coefficients.
\end{corol}
\begin{proof}
 For $r \in \mathbb{N}$ we have
\begin{align*}
\|E(t)\varphi\|_{\mathcal{F}_{0,r}^{s_1,s_2}}^2 & = \|E(t)\varphi\|_{s_1,s_2}^2+ \|E(t)\varphi\|_{L_{0,r}^2}^2\\
&\leq c\|\varphi\|_{s_1,s_2}^2 + c\int_{\re^2} |y^rE(t)\varphi|^2 \, dxdy \\
& =c \|\varphi\|_{s_1,s_2}^2 + \int_{\re^2}
|\partial_{\eta}^r (F(t,\xi, \eta)\widehat{\varphi})|^2 \, d\xi d\eta\\
&\leq  c \|\varphi\|_{s_1,s_2}^2 + \int_{\re^2}
\sum_{j=0}^{r}|c_j \partial_{\eta} ^{j}F(t,\xi,\eta)\partial_{
  \eta}^{r-j} \widehat{\varphi}|^2 \, d\xi d\eta \\
&\leq  c \|\varphi\|_{s_1,s_2}^2 + \sum_{j=0}^{r}\int_{\re^2}|c_j
( \partial_{\eta}^{j}
F(t,\xi,\eta)\partial_{\eta}^{r-j} \widehat{\varphi }|^2 \, d\xi d\eta\\ 
&\leq  c \|\varphi\|_{s_1,s_2}^2 +
\sum_{j=0}^{r}\int_{\re^2}|c_jp_{j}(t,\eta) \partial_{\eta}^{r-j}\widehat{\varphi
}|^2 \, d\xi d\eta \\
& \leq \|\varphi\|_{s_1,s_2}^2 +  Q(t)\|\varphi\|_{\mathcal{F}_{0,r}^{s_1,s_2}}^2 \\
& \leq P_{j}^2(t)\|\varphi\|_{\mathcal{F}_{0,r}^{s_1,s_2}}^2
\end{align*}
where $p_j(t,\eta)$ is a polynomial in $\eta$ of degree $j$ (whose
coefficients of the powers with parity different to that of $j$ are
zero) and $P_j$ is a polynomial in $t$ of degree $j$. Here, to estimate
the integrals where appear $p_j(t,\eta)\partial_\eta ^{r-j}\widehat
\varphi$, $j=0,1,\cdots,r$, we proceed as in \eqref{eq:2.13}.
\end{proof}

Thanks to $\mathcal{F}_{r_1,r_2 }^{ s_1,s_2}=\mathcal{F}_{
  r_1,0}^{ s_1,s_2}\cap \mathcal{F}_{0,r_2}^{s_1,s_2}$, from
Corollaries \ref{pesoeny} y \ref{pesoenx}, it follows immediately the
next corollary. 
\begin{corol}\label{pesoxy}
  Let $E$ be as in Corollaries \ref{pesoenx} and \ref{pesoeny}. If
  $s_1$, $s_2 \in \re$, $r_1=0,1,2$ and $r_2 \in \mathbb{N}$ with $s_2
  \geq \max(2r_1,r_2)$, we have
  \[\|E(t)\varphi\|_{\mathcal{F}_{r_1,r_2}^{s_1,s_2}} \leq
  P_r(t)\|\varphi\|_{\mathcal{F}_{r_1,r_2}^{s_1,s_2}},\] where
  $P_r(t)$ is a polynomial of degree $r=\max(r_1,r_2)$ with positive
  coefficients.
\end{corol}
Let us examine, now, the Cauchy problem of the linear equation
associated to the equation \eqref{eq:rBOZK} in weighted spaces $\mathcal
F_{r_1,r_2}^{s_1,s_2}$ where $r_1$ and $r_2$ are non-integers real
numbers. In this case we use in a systematic way the Stein
derivative. In fact, in the next three lemmas we will give estimates
of the Stein derivatives of some functions that appear throughout our
discussion in the last part of this section.

\begin{lema}\label{lem:2.19}
Let $b \in (0,1)$. For all $t>0$,
\begin{equation}
\mathcal{D}_x^b ( F(t,x,\eta)) =\mathcal{D}_x^b ( e^{\frac{it \eta^2
    x}{1+|x| }})  \leq C(b) t^b \eta^{2b}
\end{equation}
\end{lema}
\begin{proof}
From the homogeneous derivative definition we have
\begin{align*}
\mathcal{D}_x^b ( e^{\frac{it \eta^2 x}{1+|x| }})^2 = \int_{\re}
\frac{ |e^{\frac{it \eta^2 x}{1+|x| }}-e^{ \frac{it \eta^2 y}{1+|y|
    }}|^2}{ |x-y|^{n+2b}} \, dy =\int_{\re} \frac{|1-e^{it\eta^2 (
    \frac{ x-y}{1+|x-y| }-\frac{x}{1+|x|} )}|^2}{|y|^{1+2b}} \, dy 
\end{align*}
Let us suposse $x>0$. Then, making a change of variable we have
\begin{multline*}
  \int_{-\infty}^{x}  \frac{|1-e^{it\eta^2 ( \frac{x-y}{1+x-y
      }-\frac{x}{1+x} )}|^2}{|y|^{1+2b}} \,  dy =
  \int_{-\infty}^{\eta^2 tx}  \frac{|1 -e^{
        \frac{-iy}{(1+x-\frac{y}{\eta^2t})(1+x)}}|^2}{|y|^{1+2b}} \,
    dy =\\= (\eta^2
  t)^{2b}\int_{-\infty}^{-1} + \int_{-1}^{\eta^2 tx}  \frac{|1-e^{
        \frac{-iy}{(1+x-\frac{y}{\eta^2t})(1+x)}}|^2}{|y|^{1+2b}} \,
    dy 
\end{multline*}
Now, let us examine each of the integrals that appear in the above
equation. Since $|1-e^{ix}|<2$, we have
\begin{equation}
\int_{-\infty}^{-1} \frac{|1-e^{
    \frac{-iy}{(1+x-\frac{y}{\eta^2t})(1+x)}}|^2}{|y|^{1+2b}} \, dy
\leq  \int_{-\infty}^{-1} \frac{4}{|y|^{1+2b}} \, dy= \frac{2}{b}
\end{equation}
If $\eta^2 t x \le 1$, from the mean value theorem, 
\[\int_{-1}^{\eta^2 tx} \frac{|1-e^{ \frac{-iy }{( 1+ x-
      \frac{y}{\eta^2t}) (1+x)}} |^2}{|y|^{1+2b}} \, dy \leq
\int_{-1}^1 \frac{ y^2}{|y|^{1+2b}}\, dy = \frac{1}{1-b} \]
If $\eta^2 xt >1$, combining the arguments used above, we have
\begin{align*}
 \int_{-1}^{\eta^2 xt} \frac{|1-e^{
    \frac{-iy}{(1+x-\frac{y}{\eta^2t})(1+x)}}|^2}{|y|^{1+2b}} \, dy
&\leq \int_{-1}^{1} + \int_{1}^{\eta^2 xt} \frac{|1-e^{
    \frac{-iy}{(1+x-\frac{y}{\eta^2t})(1+x)}}|^2}{|y|^{1+2b}} \, dy \\
&\leq \frac{ 1}{1-b}+\frac{2}{b}
\end{align*}
So that,
\begin{equation}\label{eq:2.17}
\int_{-\infty}^x \frac{|e^{\frac{it \eta^2 x}{1+|x| }}-e^{\frac{it \eta^2
      y}{1+|y| }}|^2}{|x-y|^{1+2b}} \, dy \leq
\left(\frac{2}{1-b}+\frac{4}{b} \right)(\eta^2 t)^{2b}
\end{equation}
On the other hand, making a change of variable, we have
\begin{align*}
  \int_x^{\infty} \frac{|1-e^{it\eta^2 ( \frac{x-y}{1-x+y
      }-\frac{x}{1+x} )}|^2}{|y|^{1+2b}} \, dy = (\eta^2
  t)^{2b}\int^{\infty}_{\eta^2 tx} \frac{|1-e^{
      \frac{i(2x^2\eta^2t-2xy-y)}{(1-x+\frac{y}{\eta^2t})(1+x)}}|^2}{|y|^{1+2b}
  } \, dy
\end{align*}
If $t\eta^2x >1$,
\[ \int^{\infty}_{\eta^2 tx} \frac{|1-e^{
    \frac{i(2x^2\eta^2t-2xy-y)}{(1-x+\frac{y}{\eta^2t})(1+x)}}|^2}{|y|^{1+2b}
} \, dy \leq \int^{\infty}_{1}\frac{4}{|y|^{1+2b}} \, dy =
\frac{2}{b}.\] Now, if $x>0$ and $y\ge x\eta^2 t$, it has 
$1-x+\frac{y}{\eta^2t}\ge 1$ and, thence,
\begin{align*}
\left | \frac{i(2x^2\eta^2t-2xy-y)}{(1-x+\frac{y}{\eta^2t})(1+x)}
\right  | & \leq \left | \frac{2x}{1+x} \right | \left |
  \frac{x\eta^2t-y }{(1-x+\frac{y}{\eta^2t})} \right | + \left |
  \frac{y }{(1-x+\frac{y}{\eta^2t})(1+x)} \right | \\
& \leq 2|x\eta^2t-y|+|y| \leq 3|y|.
\end{align*}
Then, if $t\eta^2x <1$,
\begin{align*}\int_{x\eta^2t}^\infty \frac{|1-e^{
      \frac{i(2x^2\eta^2t-2xy-y)}{(1+x-\frac{y}{\eta^2t})(1-x)}}|^2}{|y|^{1+2b}
  } \, dy =& \int_{x\eta^2t}^1+\int_{1}^\infty \frac{|1-e^{
      \frac{i(2x^2\eta^2t-2xy-y)}{(1-x+\frac{y}{\eta^2t})(1+x)}}|^2}{|y|^{1+2b}
  } \, dy \\ \le & \int_0^1\frac{9y^2}{|y|^{1+2b}}\,  dy +
  \int_{1}^{\infty} \frac{4}{|y|^{2b+1} }\,  dy=\frac{9}{ 2(1-b)}+ \frac2b.
\end{align*}
So,
\begin{equation}\label{eq:2.18}
\int_{x}^\infty \frac{|e^{\frac{it \eta^2 x}{1+|x| }}-e^{\frac{it
      \eta^2 y}{1+|y| }}|^2}{|x-y|^{1+2b}} \, dy \leq \left(
  \frac{9}{ 2(1-b) }+\frac{2}{b} \right)(\eta^2 t)^{2b}
\end{equation}
Therefore, if $x>0$, from \eqref{eq:2.17} and \eqref{eq:2.18}, we obtain
\begin{equation*}
\int_{-\infty}^\infty \frac{|e^{\frac{it \eta^2 x}{1+|x| }}-e^{\frac{it
      \eta^2 y}{1+|y| }}|^2}{|x-y|^{1+2b}} \, dy \leq \left(
  \frac{13}{ 2(1-b) }+\frac{6}{b} \right)(\eta^2 t)^{2b}
\end{equation*}

The proof of the estimate in the case where $x<0$ is completely analogous.
\end{proof}
\begin{corol}\label{corol:2.19}
Let $b \in (0,1)$. For all $t>0$,
\begin{equation}
\mathcal{D}_x^b (\sgn(x)( F(t,x,\eta)-1)) =\mathcal{D}_x^b (
\sgn(x)(e^{\frac{it \eta^2 x}{1+|x| }} -1 ))  \leq C(b) t^b \eta^{2b}
\end{equation}
\end{corol}
\begin{proof}
Without loss of generality we suposse  $x>0$. Then
\begin{align*}
\mathcal{D}_x^b ( \sgn(x)(e^{\frac{it \eta^2 x}{1+|x| }}& -1))= \left(\int_{\re}
\frac{ |\sgn(x)(e^{\frac{it \eta^2 x}{1+|x| }}-1)-\sgn(y)(e^{ \frac{it \eta^2 y}{1+|y|
    }}-1)|^2}{ |x-y|^{n+2b}} \, dy\right)^{1/2} \\ &\le \left(\int_{\re} \frac{|e^{\frac{it
      \eta^2 x}{1+|x| }}-e^{ \frac{it \eta^2
      y}{1+|y|}}|^2}{|x-y|^{1+2b}} \, dy  \right)^{1/2}+2\left(\int_{-\infty}^0 \frac{|e^{ \frac{it \eta^2
      y}{1+|y|}}-1 |^2}{|x-y|^{1+2b}} \, dy  \right)^{1/2}.
\end{align*}
Since $|y|<|x-y|$, for $y<0$, from above lemma it follows the corollary.
\end{proof}
\begin{lema}\label{lem:2.20}
If $0<b<1$, then 
\begin{equation}\label{eq:2.19}
  \mathcal{D}_x^b \left( \frac{1}{(1+|x|)^n}\right) \leq
  \frac{C}{(1+|x|)^{1/2} }
\end{equation}
\end{lema}
\begin{proof}
It is clear that
\begin{align*}
\left(\mathcal{D}_x^b \left( \frac{1}{(1+|x|)^n}\right) \right)^2  =& \int_{\re}
\frac{\left | \frac{1}{ (1+|x|)^n}-\frac{1}{(1+|y|)^n}\right
  |^2}{|y-x|^{1+2b}} \, dy \\
\leq&\frac{1}{(1+|x|)^{2n}}\int_{\re} \frac{(2+|y|+|x|
  )^{2(n-1)}|y-x|^{1-2b}}{(1+|y|)^{2n}} \,  dy\\
\leq &\frac{1}{(1+|x|)^{2n}} \bigg( \int_{|x-y|< 1}
\frac{(2+|y|+|x|)^{2(n-1)}|y-x|^{1-2b}}{(1+|y|)^{2n}} \, dy +
\end{align*}
\begin{align*}
  \phantom{\mathcal{D}_x^b \left( \frac{1}{(1+|x|)^2}\right)^2
    =}&+\int_{|y-x|\ge 1}\frac{(2+ |y|+ |x| )^{2(n-1)}|y- x|^{1-2b}
  }{(1+|y|)^{2n} }\, dy \bigg) 
\end{align*}
Now 
\begin{align*}
\int_{x-1
}^{x+1}\frac{(2+|y|+|x|)^{2(n-1)}|y-x|^{1-2b}}{(1+|y|)^{2n}} \, dy &\leq
2(2+|x|)^{2(n-1)} \int_0^{1} y^{1-2b} \, dy  \\
& \leq C\frac{(1+|x|)^{2(n-1)}} {1-b}
\end{align*}
 On the other hand, if $b \in (0, 1/2)$,
\begin{align*}
  \int_{|y-x|\ge 1}\frac{(2+ |y|+ |x| )^{2(n-1)}|y- x|^{1-2b}
  }{(1+|y|)^{2n} }\, dy  \le \int_A +\int_B \frac{(2+ |y|+ |x| )^{
      2(n-1)}| y- x|^{1-2b}
  }{(1+|y|)^{2n} }\, dy 
\end{align*}
where $A=\{y\in \re\mid 1\le |y-x|\le |y|\}$ and $B=\{y\in \re\mid 1\le
|y-x| \text{ y } |y|<|y-x|\}$. Since
\begin{align*}
\int_A \frac{(2+ |y|+ |x|
  )^{2(n-1)}|y- x|^{1-2b} }{(1+|y|)^{2n} }\, dy \le &\int_A \frac{(2+ |y|+ |x|
  )^{2(n-1)}|y|^{1-2b} }{(1+|y|)^{2n} }\, dy \\\le &C(1+|x|)^{2(n-1)} \int_\re
(1+|y|)^{-1-2b}\, dy\\ \le &\frac Cb(1+|x|)^{2(n-1)}
\end{align*}
and
\begin{align*}
\int_B \frac{(2+ |y|+ |x|
  )^{2(n-1)}|y- x|^{1-2b} }{(1+|y|)^{2n} }\, dy \le & 2^{2b}\int_A \frac{(2+ |y|+ |x|
  )^{2(n-1)}(|y|+|x|)(1+|y|)^{-2b} }{(1+|y|)^{2n} }\, dy \\\le &C(1
+|x|)^{ 2n-1} \int_\re
(1+|y|)^{-1-2b}\, dy\\ \le &\frac Cb(1+|x|)^{2n-1},
\end{align*}
in this case we have $$ \int_{|y-x|\ge 1}\frac{(2+ |y|+ |x| )^{2(n-1)}|y- x|^{1-2b}
  }{(1+|y|)^{2n} }\, dy  \le \frac Cb(1+|x|)^{2n-1}.$$
If $b \in [1/2,1)$, then
\begin{align*}
\int_{|x-y|\ge 1}\frac{(2+ |y|+|x|)^{2(n-1)}| y-x|^{ 1-2b}} {(1+
  |y|)^{2n}} \, dy & \leq
\int_{ \re} \frac{(2+|y|+|x|)^{2(n-1)}}{(1+|y|)^{2n}} \, dy\\
& \leq C(1+|x|)^{2(n-1)}\int_{\re} \frac{1}{(1+|y|)^2} \, dy \\
& \leq C(1+|x|)^{2(n-1)}.
\end{align*}
So, $$\left( \mathcal{D}_x^b \left( \frac{1}{(1+|x|)^n}\right)
\right)^2  \le \frac C{b(1+|x|)}.$$
\end{proof}

\begin{lema}\label{steiny}
Let $0<b<1$. For $t>0$ 
\begin{equation}
\mathcal{D}_{\eta}^b ( e^{\frac{ it\xi \eta^2}{( 1+|\xi|)}}) \leq
c\left( \left( \frac{t\xi}{(1+|\xi|)}\right)^{b/2}+\left( \frac{
      t\xi}{(1 +|\xi|)}\right)^{b}|\eta|^b\right)
\end{equation}
\end{lema}
\begin{proof}
See Proposition 2 in \cite{napo}.  
\end{proof}
Now, we show two propositions extending the Corollaries
\ref{pesoenx} and \ref{pesoeny} the case where the exponent from
weight is real.
\begin{prop}\label{pesoenxr} 
  Let $E$ be as in Corollary \ref{pesoenx}.  For $s_1$, $s_2 \in
  \re$ and $0<r<\frac52$ such that $s_2 \geq 2r$, we have
  \begin{equation}\label{eq:2.21}
\|E(t)\varphi\|_{\mathcal{F}_{r,0}^{s_1,s_2}} \leq
  C(t)\|\varphi\|_{\mathcal{F}_{r,0}^{s_1,s_2}},
\end{equation}
where $C(t)$ is an increasing continuous function in $t$.
\end{prop}
\begin{proof}
  First, let us suposse $r=b$, with $0 < b < 1$. From Stein derivative
  properties and Lemmas \ref{lem:2.19} and \ref{lem:2.20}, we have
\begin{align*}
\||x|^{r} E(t) \varphi \|_{L^2} & = \|D_{\xi}^b (F\widehat{\varphi})\|_{L^2} \\
& \leq c( \|F\widehat{\varphi}\|_{L^2} + \|\mathcal{D}_{\xi}^b
(F\widehat{\varphi})\|_{ L^2} )\\
& \leq
c(\|\varphi\|_{L^2}+\|F \mathcal{D}_{\xi}^b( \widehat{\varphi})
\|_{L^2}+\|\widehat{\varphi } \mathcal{D}_{\xi}^bF\|_{L^2})\\
& \leq c (\|\varphi\|_{L^2}+ \| \mathcal{D}_{ \xi}^b(\widehat{
  \varphi}) \|_{L^2}  +\|c(b)t^b \eta^{2b} \widehat{ \varphi}\|_{L^2} )\\
& \leq c (\|\varphi\|_{L^2}+ \|\widehat{|x|^b\varphi}\|_{L^2}
+c(b)t^b\| \widehat{D_y^{2b}\varphi}\|_{L^2} )\\
& \leq c ( \|\varphi\|_{L^2}+ \||x|^b\varphi\|_{L^2} +c(b) t^b\|
D_y^{2b} \varphi\|_{L^2} )\\
& \leq c \|\varphi\|_{\mathcal{F}_{r,0}^{s_1,s_2}}
\end{align*}

Now, if $1<r<2$, $r=1+b$, for some
$0<b<1$. Then, since from Lemma \ref{interpolacion},
\begin{align*}
\||x|^b D^2_y\varphi\|^2_{L^2} \leq &c (\|\langle x \rangle ^{1+b}
\varphi \|_{L^2}^{2} +\|J_y^{2+2b} \varphi\|_{L^2}^{2} )\\
\intertext{and}
\|x D^{2b}_y\varphi\|_{L^2}  \leq &c( \|\langle x \rangle
^{1+b}\varphi\|_{L^2}^{2} +\|J_y^{2+2b} \varphi\|_{L^2}^{2}), 
\end{align*}
again, from Stein derivative properties and Lemas \ref{derivadas1},
\ref{lem:2.19} and \ref{lem:2.20}, we obtain
\begin{align*}
\||x|^{r} E(t) \varphi \|_{L^2}  = & \|D_{\xi}^b ( \partial_{\xi}(
F\widehat{ \varphi}))\|_{L^2} \\
 \leq &c( \| \partial_{\xi}( F\widehat{\varphi})\|_{L^2} +
 \|\mathcal{D}_{\xi}^b (\partial_{ \xi} (F\widehat{\varphi}))\|_{L^2} \\
 \leq &c(\|xE(t) \varphi\|_{L^2}  + \|\mathcal{D}_{\xi}^b
(\widehat{\varphi} \partial_{\xi}F )\|_{L^2} +\|\mathcal{D}_{\xi}^b
(F \partial_{\xi} \widehat{\varphi})\|_{L^2} \\  \leq &c\bigg( P_1(t)
\|\varphi \|_{\mathcal F^{s_1,s2}_{r,0}}+
\left\| \mathcal{D}_{\xi}^b\left( \widehat{\varphi}
    \frac{it\eta^2}{(1+|\xi|)^2}F \right)
\right \|_{L^2} + \|\partial_{\xi} \widehat{ \varphi}
\mathcal{D}_{\xi}^b F\|_{L^2}+ \\ 
& + \|F \mathcal{D}_{\xi}^b(\partial_{\xi} \widehat{\varphi})\| \bigg)
\\ \leq &c(t)\bigg( \|\varphi
\|_{\mathcal F^{s_1,s2}_{r,0}}+ \left\| \frac{1}{(1+|\xi|)^2}
\right\|_{\infty}\| \mathcal{D}_{\xi}^b ( \eta^2 F
\widehat{\varphi})\|_{L^2} + \\
& + \left\| \eta^2 F \widehat{\varphi} \mathcal{D}_{\xi}^b \left(
    \frac{1}{(1+|\xi|)^2} \right) \right\|_{L^2}+ \|c(b)t^b
\eta^{2b} \partial_{\xi} \widehat{\varphi}\|_{L^2} + 
\|\mathcal{D}_{\xi}^{1+b} \widehat{\varphi}\|_{L^2}\bigg) \\
 \leq &c(t)\bigg( \|\varphi
\|_{\mathcal F^{s_1,s2}_{r,0}} +\|\mathcal{D}_{\xi}^b(\eta^2
F)\widehat{\varphi}\|_{L^2}+ \|\eta^2 F \mathcal{D}_{\xi}^b
\widehat{\varphi}\|_{L^2}+ \| xD_y^{2b} \varphi\|_{L^2} \bigg)\\
 \leq &c(t,b)( \|\varphi \|_{\mathcal F^{s_1,s2}_{r,0}} +
 \|D_y^{2+2b}\varphi\|_{L^2}+\||x|^b D^2_y\varphi\|_{L^2}  + \|
 J_y^{2b+2}\varphi\|_{L^2}+ \\ &+ \|\langle x \rangle^{1+b} \varphi\|_{L^2} )\\
 \leq &c(t) \left (\|\varphi\|_{\mathcal{F}_{r,0}^{s_1,s_2}} + \|
 J_y^{2b+2}\varphi\|_{L^2}+ \|\langle x \rangle^{1+b} \varphi\|_{L^2}
\right) \\
 \leq &c(t) \|\varphi\|_{\mathcal{F}_{r,0}^{s_1,s_2}}.
\end{align*}

Finally, we suposse $2<r<\dfrac52$, or, in other words, soposse
$r=2+b$ with $0<b<\frac{1}{2}$. Then, from Lemma \ref{derivadas1},
\begin{equation}\label{segundaderivada}
  \begin{aligned}
      \||x|^{2+b} E(t) \varphi \|_{L^2}= &\|D_{\xi}^b ( \partial_{\xi}^2(F
  \widehat{\varphi}))\|_{L^2} \\
\le &\left \|D_{\xi}^b \left( \frac{2it \eta^2 \sgn(\xi)}{(1+|\xi|)^3}
  F\widehat{\varphi}\right) \right \|_{L^2} + \left\|D_{\xi}^b \left(\frac{t^2
      \eta^4}{(1+|\xi|)^4}F\widehat{\varphi}\right)\right \|_{L^2} +\\
&+ \left \|D_{\xi}^b \left(
      \frac{2it \eta^2}{(1+|\xi|)^2}F \partial_{\xi}\widehat{\varphi}
    \right) \right\|_{L^2}+  \left\|D_{\xi}^b \left( F \partial_{\xi}^2
      \widehat{\varphi}\right) \right\|_{L^2} 
  \end{aligned}
\end{equation}
Let us estimate each of the terms in the right side of the last
inequality. Proceeding as in the above case, the estimate of
fourth term is thus 
\begin{equation}\label{estimativa1}
\begin{split}
  \|D_{\xi}^b (F\partial_{\xi}^2 \widehat{\varphi})\|_{L^2} &
  \leq c(\|F\partial_{\xi}^2 \widehat{ \varphi}\|_{L^2}+\|\mathcal
  D_{\xi}^b(F\partial_{\xi}^2 \widehat{\varphi}) \|_{L^2})\\
  & \leq c(\|\partial_{\xi}^2 \widehat{\varphi} \|_{L^2}+\|F\mathcal
  D_{\xi}^b(\partial_{ \xi}^2
  \widehat{\varphi})\|_{L^2}+\|\partial_{\xi}^2 \widehat{\varphi}
  \mathcal D_{ \xi}^bF\|_{L^2}) \\
  & \leq c( \|x^2 \varphi\|_{L^2}+ \||x|^{2+b}
  \varphi\|_{L^2}+\|c(b)t^b
  \eta^{2b}\partial_{\xi}^2  \widehat{\varphi} \|_{L^2}) \\
  & \leq c( \|x^2 \varphi\|_{L^2}+ \||x|^{r}
  \varphi\|_{L^2}+c(b)t^{b}\|x^2 D_y^{2b}\varphi \|_{ L^2})\\
  & \leq c(t) \|\varphi\|_{\mathcal{F}_{r,0}^{s_1,s_2}},
\end{split}
\end{equation}
where we use the following inequality, that follows from Lemma
\ref{interpolacion} and Young inequality,
\begin{align*}
\|x^2 D^{2b}_y\varphi\|_{L^2}  \leq \|\langle x \rangle
^{2+b}\varphi\|_{L^2} +\|J^{4+2b}_y \varphi\|_{L^2} 
\end{align*}
In the same way we have for the second and the third term the following
estimates 
\begin{equation}\label{estimativa2}
  \left \|D_{\xi}^b \left(\frac{t^2 \eta^4}{(1+|\xi|)^4}F
      \widehat{\varphi} \right) \right\|_{L^2} \le c(t) \| \varphi\|_{
    \mathcal{F}_{r,0}^{s_1,s_2} }
\end{equation}
and 
\begin{equation}\label{estimativa3}
 \left \|D_{\xi}^b \left(\frac{2t\eta^2}{(1+|\xi|)^2}F \widehat{
   \varphi}\right) \right\|_{L^2} \le c(t) \| \varphi\|_{\mathcal{F}_{r,0}^{s_1,s_2}} 
\end{equation}
The treatment for the first term has the following slight difference,
thanks to the Example \ref{Ap2} and the Theorem
\ref{condicionAp}, for $0<b<\frac12$, we obtain
$$
\begin{aligned}
 \left \|D_{\xi}^b \left(\frac{2t\sgn(\xi)\eta^2}{(1+|\xi|)^3}  F
     \widehat{\varphi} \right) \right\|_{L^2}  &\leq 2t \left\||x|^b
         \mathscr{H}\left(\frac{ \eta^2}{(1+|\xi|)^3}  F
     \widehat{ \varphi} \right)^{\vee}\right\|_{L^2} \\ & \le ct\left\||x|^b
        \left(\frac{ \eta^2}{(1+|\xi|)^3}  F
     \widehat{ \varphi} \right)^{\vee}\right\|_{L^2} \\ & \le ct\left\| D_{\xi}^b
       \left(\frac{ \eta^2}{(1+|\xi|)^3}  F
     \widehat{ \varphi} \right)\right\|_{L^2} 
\end{aligned}
$$
Now, proceeding as before, we have
\begin{equation}\label{estimativa4}
 \left \|D_{\xi}^b \left(\frac{2t\sgn(\xi)\eta^2}{(1+|\xi|)^3}F \widehat{
   \varphi}\right) \right\|_{L^2} \le c(t) \| \varphi\|_{\mathcal{F}_{r,0}^{s_1,s_2}} 
\end{equation}
Then, from \eqref{segundaderivada}, \eqref{estimativa1},
\eqref{estimativa2}, \eqref{estimativa3} and \eqref{estimativa4}, it
follows \eqref{eq:2.18}, for $2<r<\frac52$.  This completes the proof
of the proposition.
\end{proof}
\begin{prop}\label{pesoenyr} 
  Let $E$ be as in Corollary \ref{pesoenx}.  For $s_1$, $s_2 \in
  \re$ and $0<r$ such that $s_2 \geq r$, we have
  \begin{equation}\label{eq:2.27}
\|E(t)\varphi\|_{\mathcal{F}_{0,r}^{s_1,s_2}} \leq
  C(t)\|\varphi\|_{\mathcal{F}_{0,r}^{s_1,s_2}},
\end{equation}
is an increasing continuous function in $t$.
\end{prop}
\begin{proof}
Let $n\in\nat$ and $0<b<1$ be such that $r=n+b$. Then,
\begin{equation}\label{eq:2.28}
\begin{aligned}
\||y|^{r} E(t) \varphi \|_{L^2}  &\le c(\|\phi\|_{\mathcal{F}_{0,r}^{s_1,s_2}} +\| \mathcal D_{\eta}^b
(\partial_\eta^n(F\widehat{\varphi}))\|_{L^2}) \\ &\leq c\left(
\|\phi\|_{\mathcal{F}_{0,r}^{s_1,s_2}} + \sum_{m=0}^n \| \mathcal D_{\eta}^b (\partial_\eta
^{m}F \partial_\eta^{n-m}\widehat{\varphi}) \|_{L^2}\right)
\end{aligned}
\end{equation}
Let us examine each of the terms in the sum in the right side from
inequality above. We suposse without loss generality that $m$ is even,
that means, $m=2j$. From Lemma \ref{derivadas2} and the Stein
derivetive properties it follows that
$$
\begin{aligned}
\| \mathcal D_{\eta}^b (\partial_\eta
^{m}F \partial_\eta^{n-m}\widehat{\varphi}) \|_{L^2} \le
&\sum_{k=0}^{j}a_{k} \left\| \mathcal D_{\eta}^b \left( \eta^{2k} \left(
    \frac{2i t \xi}{1+ |\xi|} \right)^{j+ k}F \partial_
  \eta^{ n-m}\widehat{\varphi}\right)
  \right\|_{L^2} \\ \le & c(t)\sum_{k=0}^{j} \left( \left\| \mathcal
    D_{\eta}^bF  \eta^{2k}  \partial_  \eta^{ n-m}\widehat{\varphi}
  \right\|_{L^2} + \left\| F\mathcal
    D_{\eta}^b(\eta^{2k}  \partial_\eta^{ n-m} \widehat{\varphi})
  \right\|_{L^2} \right)\\ \le & c(t)\sum_{k=0}^{j} \left( \left\| \mathcal
      \eta^{2k+b}  \partial_\eta^{ n-m}\widehat{\varphi}
  \right\|_{L^2} + \left\| \mathcal
    D_{\eta}^{b}(\eta^{2k} \partial_\eta^{ n-m} \widehat{\varphi})
  \right\|_{L^2} \right).
\end{aligned}
$$
Since, from Lemma \ref{interpolacion}
\begin{align*}
  \|\eta^{2k+b} \partial_\eta^{n -m} \widehat\varphi\|_{L^2} & \leq
  \|\langle \eta\rangle ^{2k+b} J_\eta^{n-m} \widehat\varphi \|_{L^2}
  \leq
  \|\langle \eta\rangle ^{m+b} J_\eta^{n-m} \widehat\varphi \|_{L^2} \\
  & \leq \|J_y^{n+b} \varphi\|_{L^2}^{\frac{m+b}{n+b}}  \|\langle
  \eta\rangle^{n} \varphi\|_{L^2}^{\frac{ n-m}{n+b}} \\
  & \leq c (\|\langle y\rangle ^{n+b}\varphi\|_{L^2}+\|J_y^{n+b} \varphi\|_{L^2})
\intertext{and}
  \left\| \mathcal D_{\eta}^{b}(\eta^{2k} \partial_\eta^{ n-m}
    \widehat{\varphi}) \right\|_{L^2} & \leq \| J_\eta^{b} \langle \eta\rangle
  ^{2k} J_\eta^{n-m} \widehat\varphi \|_{L^2}  \\
  & \leq \|J_y^{n-m+ 2k+b} \varphi\|_{L^2}^{\frac{2k}{2k+b}} \|\langle
  \eta\rangle^{2k +b } J_\eta^{n-m} \varphi\|_{L^2}^{\frac{ b}{2k+b}}
  \displaybreak[3]\\
  & \leq \|J_y^{n +b} \varphi\|_{L^2}^{\frac{2k}{2k+b}} \|\langle
  \eta\rangle^{m +b } J_\eta^{n-m} \varphi\|_{L^2}^{\frac{ b}{2k+b}}\\
  & \leq \|J_y^{n +b} \varphi\|_{L^2}^{(\frac{2k}{2k+b} + \frac{(n-m)
      b } {(n+b)(2k+b) } )} \|\langle \eta\rangle^{m +b } J_\eta^{n-m}
  \varphi \|_{L^2}^{ \frac{(m+b) b}{(n+b) (2k+b)} }\\
  & \leq c (\|\langle y\rangle ^{n+b}\varphi\|_{L^2}+\|J_y^{n+b}
  \varphi\|_{L^2}),
\end{align*}
we have
$$
\begin{aligned}
\| \mathcal D_{\eta}^b (\partial_\eta
^{m}F \partial_\eta^{n-m}\widehat{\varphi}) \|_{L^2} &\le
c (\|\langle y\rangle ^{n+b}\varphi\|_{L^2}+\|J_y^{n+b}
  \varphi\|_{L^2})\\ &\le c\|\varphi\|_{\mathcal{F}_{0,r}^{s_1,s_2}}
\end{aligned} 
$$
Then, from \eqref{eq:2.28}, it follows the proposition.
\end{proof}
From the two previous  propositions we obtain the following corollary
analogous to Corollary \ref{pesoxy}, in the same way that we obtained
this.
\begin{corol}\label{pesoxyr}
  Let $E$ be as in Corollary \ref{pesoenx}. If 
  $s_1$, $s_2 \in \re$, $0\le r_1< 5/2$ and $r_2 \ge 0$ with $s_2
  \geq \max(2r_1,r_2)$, we have
  \[\|E(t)\varphi\|_{\mathcal{F}_{r_1,r_2}^{s_1,s_2}} \leq
  c(t)\|\varphi\|_{\mathcal{F}_{r_1,r_2}^{s_1,s_2}},\] where
  $c(t)$ is an increasing continuous function in $t$.
\end{corol}

\subsection{Well-posedness of \eqref{eq:rBOZK}}
Thanks to the  discussion in the previous subsection we have the
following theorem.
\begin{teo}\label{pvireales}
  Let $s_1$ and $s_2$ be positive real numbers, and $r_1$ and $r_2$
  nonegative real numbers such that $r_1<\frac{5}{2}$, $s_2\ge \min(
  2r_1, r_2)$ and $\frac1{s_1} +\frac1{s_2} <2$. For all $\varphi \in
  \mathcal{F}_{r_1,r_2}^{s_1,s_2}$, there exists $T>0$, depending on
  $\|\varphi\|_{\mathcal{F}_{r_1,r_2}^{s_1,s_2}}$, and a unique $u \in
  C([0,T] ; \mathcal{F}_{r_1,r_2}^{s_1,s_2} )$ solution of the initial
  value problem associated to \eqref{eq:rBOZK}. \par Moreover, the map
  $\psi\mapsto v$, $v$ solution of \eqref{eq:rBOZK} with initial
  contition $\psi$, is continuous on the open set $
  \mathcal{F}_{r_1,r_2}^{s_1,s_2}$, containing $\varphi$, with values
  in $C([0,T] ; \mathcal{F}_{r_1,r_2}^{s_1,s_2} )$.
\end{teo}
\begin{proof}
It is clear that \eqref{eq:rBOZK} is equivalent to the integral equation
\begin{equation}\label{ecuacionintegral}
u= E(t)\varphi +\int_0^t E(t- \tau) B(u^n(\tau)) \, d\tau,
\end{equation}
where
\begin{equation}\label{operadores}
  E(t)=e^{tB}\qquad\text{ and}\qquad B=- \partial_x
  (1+\mathscr{H}\partial_x)^{-1}\partial_{ y}^2. 
\end{equation}
So, we will show that this equation has solution in $C([0,T];
\mathcal{F}_{r_1,r_2}^{s_1,s_2} )$, for $T$ small enough. Let $\Phi$
be defined by
$$ \Phi u= E(t)\varphi +\int_0^t E(t- \tau) B(u^n(\tau)) \, d\tau,$$
for each function $u\in C([0,T] ; \mathcal{F}_{r_1,r_2}^{s_1,s_2}
)$. Let us see that, for some $T$, $\Phi$ is a contraction in the 
space
\begin{equation}
\mathcal{X}(T,M)= \{ u \in C([0,T];  \mathcal{F}_{r_1,r_2}^{s_1,s_2} ) \mid \|E(t)\varphi-u(t)\|_{\mathcal{F}_{r_1,r_2}^{s_1,s_2}} \leq M\},
\end{equation}
where $M$ is a arbitrary positive real number. It is clear that
$\mathcal{X}(T,M)$ is a complete metric space with the metric
$d_{T,M}$ defined by
\[d_{T,M}(u,v)= \sup \limits_{t \in [0,T]}
\|u(t)-v(t)\|_{\mathcal{F}_{r_1,r_2}^ {s_1,s_2}}.\] We will choose $T$ such
that $\Phi$ maps $\mathcal{X}(T,M)$ in itself. Thanks to Corollaries \ref{algban}, \ref{acotop} and \ref{pesoxyr}, for $u\in
\mathcal X(T,M)$, we have
\begin{align*}
  \|\Phi(u)(t)-E(t)\varphi\|_{\mathcal{F}_{r_1,r_2}^{s_1,s_2}} & = \left\|\int_0^t E(t-\tau)B(u^n(\tau)) \, d\tau \right\|_{\mathcal{F}_{r_1,r_2}^{s_1,s_2}} \\
  & \leq \int_0^t \|E(t-\tau)B(u^n(\tau))\|_{\mathcal{F}_{r_1,r_2}^{s_1,s_2}} \, d\tau \\
  & \leq \int_0^t c(t-\tau) \|u^n(\tau)\|_{\mathcal{F}_{r_1,r_2}^{s_1,s_2}} \, d\tau \\
  & \leq c(T) \int_0^t \|u(\tau)\|_{\mathcal{F}_{r_1,r_2}^{s_1,s_2}}^n \, d\tau \\
  & \leq c(T) \int_0^t (M+\|\varphi\|_{\mathcal{F}_{r_1,r_2}^{s_1,s_2}})^n \, d\tau   \\
  & \leq c(T)(M+\|\varphi\|_{\mathcal{F}_{r_1,r_2}^{s_1,s_2}})^n T,
\end{align*} 
where $c$ is an increasing continuous function.  Then, if we choose $T$
in such a way that
\[Tc(T) \leq \frac{M}{(M+\|\varphi\|_{\mathcal{F}_{r_1,r_2}^{s_1,s_2}})^n },\]
$\Phi(u) \in \mathcal{X}(T,M)$.
On the other hand, since
\begin{align*}
\|\Phi(u)(t)-\Phi(v)&(t)\|_{\mathcal{F}_{r_1,r_2}^{s_1,s_2}}  \leq
\int_0^t \|E(t-\tau)B(u^n (\tau))-v^n(\tau))\|_{\mathcal{F}_{r_1,r_2}^{s_1,s_2}} \, d \tau \\
& \leq c(T)\int_0^t \|u^n(\tau)-v^n(\tau)) \|_{\mathcal{F}_{r_1,r_2}^{s_1,s_2}} \, d \tau \\
& \leq \tilde c(T) \int_0^t \|u(\tau)-v(\tau)\|_{\mathcal{F}_{r_1,r_2}^{s_1,s_2}}
(\|u(\tau)\|_{\mathcal{F}_{r_1,r_2}^{s_1,s_2}}^{n-1} +\|v(\tau)\|_{\mathcal{F}_{r_1,r_2}^{s_1,s_2}}^{n-1}) \, d\tau \\
& \leq \tilde c(T) (M+\|\varphi \|_{\mathcal{F}_{r_1,r_2}^{s_1,s_2}} )^{n-1}T d_{T,M}(u,v)
\end{align*}
for some increasing continuous function $\tilde c$, if we take $T$
such that it also satisfies
\[\tilde c(T)(M+\|\varphi \|_{\mathcal{F}_{r_1,r_2}^{s_1,s_2}})^{n-1} T < 1,\]
we have, indeed, that $\Phi$ is a contraction in $\mathcal X(T,M)$.
Then, from the Banach fixed point theorem, there is a function $u \in
\mathcal{X}(T,M)$ solution to \eqref{ecuacionintegral}, and thereby
solution to \eqref{eq:rBOZK}.\par
Finally, if $u$ and $v$ are solutions to \eqref{eq:rBOZK} in the
interval $[0,T]$, we have \[\|u(t)-v(t)\|_{\mathcal{F}_{r_1,r_2}^{s_1,s_2}}\le
c(t)\|\varphi-\psi\| + \int_0^t
M^{n-1}c(t-\tau)\|u(\tau)-v(\tau)\|_{\mathcal{F}_{r_1,r_2}^{s_1,s_2}}
\, d\tau. \] From this last inequality and the Gronwall lemma it
follows the theorem.
\end{proof}
\subsection{Unique continuation of the solutions to \eqref{eq:rBOZK}}

Let us see now that if $r_1=5/2$ there is not persistence of the solutions to
\eqref{eq:rBOZK}. More precisely we have the next theorem.
\begin{teo}\label{continuacionunica}
  Suposse that the Theorem \ref{pvireales} conditions are satisfied
  for $s_1$, $s_2$, $r_1$ and $r_2$, with $s_2\ge 5$. Let $u \in C([0,T]
  ; \mathcal{F}_{r_1,r_2}^{s_1,s_2})$ be solution to \eqref{eq:rBOZK}
  such that $\int_\re \partial_y^2u(0,x,y)\, dx \geq 0$,
  for all $y\in \re$. If for two times $t_1=0 < t_2 < T$ it has that
  $u(t_j) \in \mathcal{F}_{ \frac{5}{2},r_2}^{s_1, s_2}$, $j=1,2$,
  then $u$ is identically zero.
\end{teo}
\begin{proof}
  Let $u \in C([0,T] : \mathcal{F}_{2,r_2}^{s_1,s_2})$ be as in the
  statement of the theorem. Let $\varphi(x,y)=u(0,x,y)$ and
  $v=u^2$.\par
  First we will carefully examine the terms of the right side of the
  equation obtained by multiplying $x^{2}$ in both sides of
  \eqref{ecuacionintegral}. \par For the first term that appear there,
  from Lemma \ref{derivadas1}, we have
\begin{align*}
  \partial_{\xi}^2(F\widehat{\varphi}) & = F \left [ \frac{-2it \eta^2 \sgn(\xi)}{(1+|\xi|)^3}
    \widehat{\varphi}- \frac{t^2 \eta^4}{( 1+|\xi|)^4} \widehat{
      \varphi} + \frac{2it \eta^2}{(1+|\xi|)^2} \partial_{\xi}
    \widehat{ \varphi}+ \partial_{\xi}^2 \widehat{\varphi}\right ]\\
  &= B_1+B_2+B_3+B_4,
\end{align*}
where each $B_i$, $i=1,\cdots,4$, are the summands obtained by
distributing $F$ in the sum in brackets on the right side of the
equation above. Since the estimates (\ref{estimativa1}),
(\ref{estimativa2}) and (\ref{estimativa3}) are valid for $0<b<1$,
$D_{\xi}^{1/2}B_i\in L^2(\re^2)$, $i=2,$ 3 and 4. Now,
for $B_1$ we have
\begin{equation}\label{c2}
\begin{split}
  B_1=&\frac{-2it  \eta^2}{(1+|\xi|)^3}\widehat{ \varphi}\sgn(\xi)(F-1)
  - {2it \sgn( \xi)\eta^2 }\left(\frac 1{(1+|\xi|)^3}-1 \right) \widehat{
    \varphi} -\\ &- 2it \sgn( \xi)\eta^2 \widehat{\varphi}  \\
   =& C_1+C_2 +C_3 
\end{split}
\end{equation}
Reasoning in the same way as we obtained \eqref{estimativa1} it shows
that $D_{\xi}^{\frac{1}{2}}C_2\in L^2(\re^2)$. It is slightly more
difficult to see that $D_{ \xi}^{ \frac{1}{2}} C_1 \in L^2(\re^2)$,
but thanks to Corollary \ref{corol:2.19}, we can do this with the
same argument.  Then,
\begin{equation}\label{eq:3.22}
D^{1/2}_\xi\left( B_1- 2it \eta^2
  \sgn(\xi) \widehat{\varphi}\right) \in
L^2(\re^2).
\end{equation}
So, we get
\begin{equation}\label{eq:3.23}
D^{1/2}_\xi \left((x^2 E(t)
  \varphi)\,\widehat{ } + 2it \eta^2  \sgn(\xi)
  \widehat{\varphi}\right) \in L^2(\re^2). 
\end{equation}
\par
To examine the term of the integral, we will examine the second
derivative with respect to $\xi$ of the Fourier transform of the
expression that appears inside the integral. For this purpose, we have
\begin{align*}
  \partial_\xi^2\left(\frac{i\xi F\widehat v}{1+|\xi|} \right)=&
  F(t,\xi,\eta)  \biggl ( \frac{2t|\xi|\eta^2}{(1+|\xi|)^4}
  \widehat{v}- \frac{2t\eta^2}{(1+|\xi| )^4}
  \widehat{v}-\frac{it^2\eta^4}{(1+| \xi|)^5} \xi\widehat{v}- \\ &- 
  \frac{2t\xi \eta^2}{(1+|\xi|)^3} \partial_{\xi} \widehat{v}-
  \frac{2i\sgn(\xi)}{(1+|\xi|)^3}  \widehat{v} + \frac{2i}{(1+|\xi|)^2} \partial_{\xi} \widehat{v}+
  \frac{i\xi}{1+|\xi|}   \partial_{\xi}^2\widehat{v} \biggr )\\
  = &A_1+ \cdots+ A_7
\end{align*}
Proceeding as before, we can see that $D_{\xi}^{1/2}(A_i) \in
C([0,T]: L^2( \re^2))$, for $i\ne 5$, and that 
\begin{equation}\label{mm1}
  \left ( D_{\xi}^{1/2} (A_5 -2i\sgn(\xi) \widehat v ) \right) \in C([0,T]: L^{2}(\re^2)).  
\end{equation}
 Therefore, 
\begin{equation}\label{mm}
 D_{\xi}^{1/2}\left ( \left(x^2\int_0^t E(t-\tau)B(u^n(\tau)) \, d\tau
  \right)\, \widehat{ } +2i\sgn(\xi) \int_0^t\widehat v d\tau \right) \in C([0,T]: L^{2}(\re^2)).  
\end{equation}
\par
So, from \eqref{eq:3.23} and \eqref{mm}, we obtain
\begin{equation}\label{eq:3.24}
 D_{\xi}^{1/2}\left ( \widehat{x^2u(t)} + 2it \eta^2 \sgn(\xi)
  \widehat{\varphi} +
   2i\sgn( \xi) \int_0^t\widehat v d\tau \right) \in  L^{2}(\re^2),  
\end{equation}
for all $t \in [0, T]$. Since $u(t_2) \in \mathcal{F}_{\frac{5}{2},r_2}^{s_1,s_2}$ 
\[D_{\xi}^{\frac{1}{2}} \left[  \sgn(\xi) \left ( 2t_2
    \eta^2 \widehat{ \varphi} + \int_{0}^{t_2} \widehat{ v} \, d \tau
  \right ) \right] \in L^2(\re^2).\]
In particular,
\[
D_{\xi}^{\frac{1}{2}} \left[  \sgn(\xi)  \left ( 2t_2
     \widehat{\partial_y^2  \varphi}^x + \int_{0}^{t_2} \widehat{ v}^x \, d \tau
  \right ) \right] \in L^2(\re)
\]
(in $\xi$) for almost every $y\in\re$, where $\widehat{\phantom a }^x$
denotes the Fourier transform only on $x$. By the Theorem \ref{lim}
\[  2t_2\widehat{\partial_y^2 \varphi}^x(0,y) + \int_{0}^{t_2}
  \widehat{v}^x(\tau, 0, y) \, d \tau=0, \] 
for almost every $y\in\re$. In view that
\[\widehat{\partial_y^2 \varphi}^x (0,y) = \int_{\re}
\partial_y^2 \varphi( x, y) \, dx\ge 0\]
for almost every $y$,  
\[ \int_{0}^{t_2} \widehat{ v}^x \, d \tau =\int_{0}^{t_2} \int_{-\infty}^{\infty} u^2(\tau, x,y) dx  \, d
\tau=0,\]
for almost every $y$. In other words, $u \equiv 0$. 
\end{proof}
\bibliographystyle{acm}
\bibliography{mybiblio}

\end{document}